\documentclass[11pt]{article}
\usepackage{microtype}
\usepackage[textwidth=6.2in,textheight=8.5in,centering]{geometry}
\usepackage{graphicx} 
\usepackage{epsfig}
\usepackage{caption}
\usepackage{amsfonts}
\usepackage{amssymb}
\usepackage{amsmath}
\usepackage{amsthm}
\usepackage{needspace}
\usepackage[normalem]{ulem}
\usepackage{enumitem}
\usepackage{graphics}
\usepackage[hidelinks]{hyperref}
\usepackage{float}
\usepackage{pict2e}
\usepackage{cite}
\usepackage{subfigure}
\usepackage[subnum]{cases}
\usepackage{multirow}
\usepackage{comment}
\usepackage{xcolor}
\usepackage{cleveref}
\usepackage{authblk}
\usepackage{tikz}
\usepackage{array} 
\usepackage{booktabs} 
\renewcommand{\arraystretch}{1.5}
\usepackage{color}

\newtheorem{theorem}{Theorem}[section]
\newtheorem{lemma}[theorem]{Lemma}

\newtheorem{claim}{Claim}
\theoremstyle{definition}
\newtheorem{definition}[theorem]{Definition}

\newcommand{\DP}{\mathrm{DP}}

\makeatletter

\title{Degree-truncated DP-colouring of $\{K_4,K_{2,t}\}$-minor-free graphs}

\author[1]{Huan Zhou}

\affil[1]{\small School of Science, Jimei University, Xiamen, China. {Emails: \texttt{huanzhou@zjnu.edu.cn}}}

\begin{document}
\maketitle
\vspace{-2.4em}

\begin{abstract}
A graph $G$ is degree-truncated DP-$k$-colourable
(respectively, degree-truncated $k$-choosable) if it is
DP-$f$-colourable (respectively, $f$-choosable), where
$f(v)=\min\{k,d_G(v)\}$ for every vertex $v$.
Hutchinson proved that every $2$-connected maximal outerplanar
graph other than the triangle is degree-truncated $5$-choosable.
Lo, Wang, Zhou and Zhu strengthened this result by proving that
every $2$-connected $K_{2,4}$-minor-free graph other than cycles
and complete graphs is degree-truncated DP-$5$-colourable.

Outerplanar graphs are precisely the $\{K_4,K_{2,3}\}$-minor-free
graphs. We study how the truncation threshold grows when
$K_{2,3}$ is replaced by $K_{2,t}$, while $K_4$-minor remains excluded.
A bound independent of $t$ is impossible, as shown by the
family $\{K_{2,n}| n\ge 2\}$.
Let $\mathcal S_t$ be the class of $2$-connected
$\{K_4,K_{2,t}\}$-minor-free graphs other than cycles.
Writing $\chi^{\text{\sout{d}}}_{\mathrm{DP}}(\mathcal S_t)$
for the least truncation threshold valid throughout this class,
we prove that, for every integer $t\ge3$,
\[
\left\lfloor2\sqrt{2t}\right\rfloor
\le \chi^{\text{\sout{d}}}_{\mathrm{DP}}(\mathcal S_t)
\le \left\lceil4\sqrt{2t}\right\rceil.
\]
\end{abstract}



\section{Introduction}\label{sec:main}

A list assignment $L$ assigns each vertex $v$ a set $L(v)$ of colours. 
A graph is \emph{$k$-choosable} if every such assignment with
$|L(v)|\ge k$ for all $v$ admits a proper colouring $c$ satisfying
$c(v)\in L(v)$ for all $v$.
Here a colouring is proper if adjacent vertices receive different colours.
A greedy argument shows that a graph $G$ can be coloured from any
lists $L(v)$ satisfying $|L(v)|\ge d_G(v)+1$ for every vertex $v$.
A graph is \emph{degree-choosable} if it can be coloured from every
list assignment with $|L(v)|\ge d_G(v)$ at every vertex.
Borodin~\cite{Borodin} and, independently,
Erd\H{o}s, Rubin and Taylor~\cite{ERT} proved that a connected graph
is degree-choosable if and only if it is not a \emph{Gallai-tree},
that is, a connected graph whose blocks are complete graphs or odd
cycles.

Degree-truncated choosability asks whether the lists at vertices of
large degree can be reduced to a fixed size.
Such bounds could be useful in inductive colouring arguments.
After part of a graph is coloured, a vertex with few uncoloured
neighbours may have lost many colours to its coloured neighbours,
while a vertex with many uncoloured neighbours may still have  a large number of colours available.
A graph $G$ is \emph{degree-truncated $k$-choosable} if it can be coloured from every
list assignment satisfying
\[
 |L(v)|\ge\min\{k,d_G(v)\}\qquad \text{for each vertex } v\in V(G).
\]
Richter asked whether every $3$-connected non-complete planar graph
is degree-truncated $6$-choosable; see~\cite{Hutchinson}.
Degree-truncated choosability of planar graphs has been studied
further in~\cite{ZZZ,JXXZ,XZZZ}. The least integer $k$ such that
every $3$-connected non-complete planar graph is degree-truncated
$k$-choosable is now known to be either $10$ or $11$~\cite{XZZZ}.

Motivated by Richter's question, Hutchinson~\cite{Hutchinson} proved that
every $2$-connected maximal outerplanar graph other than the triangle
is degree-truncated $5$-choosable.  Lo, Wang, Zhou and Zhu~\cite{LWWZ} strengthened Hutchinson's result  to a larger graph class and to DP-colouring. 
They proved that every $2$-connected $K_{2,4}$-minor-free graph other than cycles and complete graphs is degree-truncated DP-$5$-colourable.  
The outerplanar case was also proved by Dai, Hu, Li and
Maezawa~\cite{DHLM}: every $2$-connected outerplanar
graph other than a cycle is degree-truncated DP-$5$-colourable.

DP-colouring (also called correspondence colouring) was introduced
by Dvo\v{r}\'{a}k and Postle~\cite{DP}. A \emph{DP-cover} of $G$
is a pair $(L,M)$. Here $L$ assigns pairwise disjoint finite sets
$L(v)$ of colours to the vertices of $G$, and $M=(M_{uv}:uv\in E(G))$
assigns a matching $ M_{uv}\subseteq L(u)\times L(v)$ to each edge $uv\in E(G)$.

An \emph{$(L,M)$-colouring} chooses one colour at each vertex so that
\[
 c(v)\in L(v)\quad(v\in V(G)),\qquad
 (c(u),c(v))\notin M_{uv}\quad(uv\in E(G)).
\]
An DP-cover $(L,M)$ of $G$ is viewed as a graph with vertex set $\cup_{v \in V(G)} L(v)$ and edge set $\cup_{uv \in E(G)}M_{uv}$.
An \emph{ $(L,M)$-colouring} of $G$ is an independent set of $(L,M)$ meeting each $L(v)$
in exactly one vertex.
List colouring is the special case in which colours in different
lists have labels, and each edge matches exactly the colours with
a common label in its endpoint lists.

For $f:V(G)\to\mathbb Z_{\ge0}$, a DP-cover $(L,M)$ is an
\emph{$f$-cover} if $|L(v)|\ge f(v)$ for every vertex $v$.
The graph $G$ is \emph{DP-$f$-colourable} if every $f$-cover
admits an $(L,M)$-colouring. It is
\emph{degree-truncated DP-$k$-colourable} if this holds for $f(v)=\min\{k,d_G(v)\}$.

Write $\chi^{\text{\sout{$\mathrm{d}$}}}_{\DP}(G)$ for the least
$k$ such that $G$ is degree-truncated DP-$k$-colourable, taking the
value $\infty$ if no such $k$ exists. For a graph class $\mathcal G$,
set
\[
 \chi^{\text{\sout{$\mathrm{d}$}}}_{\DP}(\mathcal G)
 =\sup_{G\in\mathcal G}\chi^{\text{\sout{$\mathrm{d}$}}}_{\DP}(G).
\]

A  connected
graph is DP-$f$-colourable for $f(v)=d_G(v)$ if and only if it is not a
\emph{GDP-tree}, whose blocks are complete graphs or cycles~\cite{BKP,KO}.
Thus GDP-trees have no finite degree-truncated DP threshold.

Excluding GDP-trees alone does not yield a uniform bound on the
degree-truncated DP-chromatic number, even for outerplanar graphs.
For any $k\ge2$, attach $k$ leaves to one vertex of a connected
outerplanar graph that is not a GDP-tree. The resulting graph is not degree-truncated DP-$k$-colourable. To exclude such graphs,  we restrict to  $2$-connected graphs.

For $k\ge2$, give the two vertices of degree $k^2$ in $K_{2,k^2}$ disjoint lists $\{a_1,\ldots,a_k\}$ and $\{b_1,\ldots,b_k\}$. Give its other vertices the lists $\{a_i,b_j\}$, one for each $1\le i,j\le k$. Every choice of colours at the two high-degree vertices leaves a common neighbour with no available colour. Hence $K_{2,k^2}$ is not degree-truncated $k$-choosable, and hence not degree-truncated DP-$k$-colourable.
These graphs are $2$-connected and $K_4$-minor-free, so there is not constant upper bound for the degree-truncated choice number of  
$K_4$-minor graphs.  
We therefore also exclude $K_{2,t}$-minor. For $t\ge3$, let $\mathcal S_t$
be the class of $2$-connected $\{K_4,K_{2,t}\}$-minor-free graphs
other than cycles. Among $2$-connected $K_4$-minor-free graphs,
GDP-trees are precisely cycles.

Our main result determines the optimal order of
$\chi^{\text{\sout{$\mathrm{d}$}}}_{\DP}(\mathcal S_t)$ as a function
of $t$.

\begin{theorem}\label{thm:main}
For every integer $t\ge3$,
\[
 \lfloor2\sqrt{2t}\rfloor
 \le\chi^{\text{\sout{$\mathrm{d}$}}}_{\DP}(\mathcal S_t)
 \le\lceil4\sqrt{2t}\rceil.
\]
In particular, $\chi^{\text{\sout{$\mathrm{d}$}}}_{\DP}(\mathcal S_t)
=\Theta(\sqrt t)$.
\end{theorem}

In fact, we prove the following sharper bounds.

\begin{theorem}\label{thm:main-precise}
For every integer $t\ge3$, with $r=\lfloor(t-1)/2\rfloor$,
\[
 1+\lfloor\sqrt{16r+4}\rfloor
 \le\chi^{\text{\sout{$\mathrm{d}$}}}_{\DP}(\mathcal S_t)
 \le \lceil8\sqrt r\rceil.
\]
\end{theorem}

Theorem~\ref{thm:main} follows from these bounds: for $t\ge5$,
$\sqrt{8t}\le\sqrt{16r+16}<1+\sqrt{16r+4}$, and taking floors
gives its lower bound. The cases $t=3,4$ follow directly, and
its upper bound follows from $2r<t$.

For $t=3,4$, the lower bound in Theorem~\ref{thm:main-precise}
and the theorem of Lo, Wang, Zhou and Zhu~\cite{LWWZ}
give $\chi^{\text{\sout{$\mathrm{d}$}}}_{\DP}(\mathcal S_t)=5$.
Table~\ref{tab:small-values} compares these exact values with the
general bounds.

\begin{table}[H]
\centering
\renewcommand{\arraystretch}{1.1}
\begin{tabular}{@{}cccc@{}}
\toprule
$t$ & Exact value & Theorem~\ref{thm:main} & Theorem~\ref{thm:main-precise} \\
\midrule
$3$ & $5$ & $[4,10]$ & $[5,8]$ \\
$4$ & $5$ & $[5,12]$ & $[5,8]$ \\
\bottomrule
\end{tabular}
\caption{Exact thresholds and intervals given by the general bounds for $t=3,4$.}
\label{tab:small-values}
\end{table}

Zhou, Zhu and Zhu~\cite{ZZZ} proved   general bounds 
for degree-truncated DP-colouring in proper minor-closed families
under suitable connectivity assumptions. In particular, for each
fixed $t$, their theorem gives a uniform threshold for $2$-connected
$K_{2,t}$-minor-free graphs that are not GDP-trees.
For $\mathcal S_t$, setting $s=2$ in their
proof~\cite{ZZZ} gives
\[
 \chi^{\text{\sout{$\mathrm{d}$}}}_{\DP}(\mathcal S_t)
 \le 4096t^2(t+1)+16t=O(t^3).
\]
For the $K_4$-minor-free subclass $\mathcal S_t$,
Theorem~\ref{thm:main} improves this cubic estimate to a square-root
bound and shows that this order is best possible.

Our upper-bound proof follows the method from our earlier joint
work with Lo, Wang and Zhu~\cite{LWWZ}. We take
advantage of the graph representation of a DP-cover: edges between lists in the DP-cover represent forbidden
colour pairs. A colouring is therefore an independent set
containing one colour from each list. During the induction, we
delete a vertex with two neighbours in the reduced graph and
add edges between their lists for colour pairs that cannot
extend to the deleted vertex.

For a degree-two vertex with a two-colour list in a DP-cover, these
forbidden pairs form a subgraph of $K_{2,2}$. This observation
leads to constraints built from matchings and such small
subgraphs. We use weight of vertices to bound how many colours a choice at  its neighbour vertex can forbid.  We extend this approach by introducing
sets $A_{xy}(x)$ that record forbidden colours as part of the
constraints. The exclusion of a $K_{2,t}$ minor bounds the number of internally vertex-disjoint paths between any two vertices in terms of $t$.
We refer to these path counts as \emph{path multiplicities}.
These multiplicities control both the sizes of the forbidden-colour sets
and the number of small subgraphs.

\section{A square-root lower bound}\label{sec:lower}

In DP-cover constructions, we use local colour labels.
Labelling $L(v)$ by a set $A$ means taking
$L(v)=\{(v,a):a\in A\}$; we write $a$ for $(v,a)$ when the
vertex is clear. Thus equal labels in different lists denote
distinct colours. Matching equal labels on $uv$ means matching
$(u,a)$ with $(v,a)$ for each common label $a$.

Let $D(u,v)$ be the graph in Figure~\ref{fig:D-cover}(a). We call
$u,v$ its \emph{terminals} and $p,z,q$ its \emph{internal vertices}.

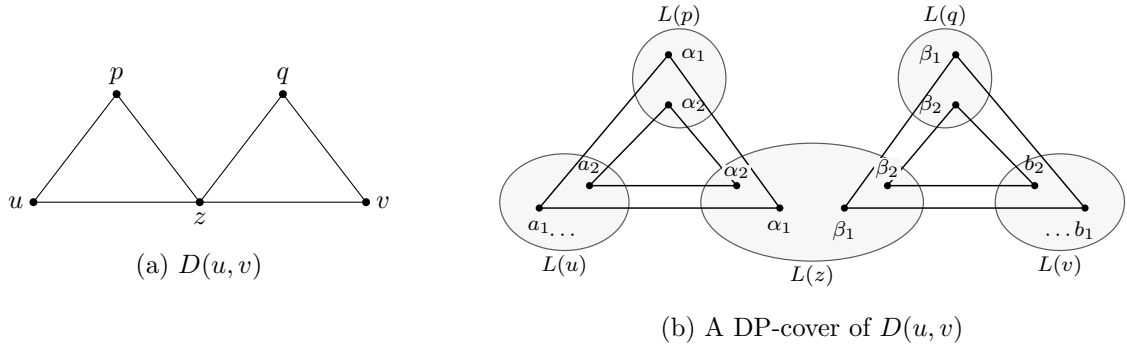
\begin{figure}[H]
\centering
\begin{minipage}[c]{.36\textwidth}
\centering
\begin{tikzpicture}[x=11mm,y=11mm,
 every node/.style={font=\small}]
\coordinate (u) at (-2,0);
\coordinate (p) at (-1,1.3);
\coordinate (z) at (0,0);
\coordinate (q) at (1,1.3);
\coordinate (v) at (2,0);
\draw (u)--(p)--(z)--(u) (v)--(q)--(z)--(v);
\foreach \w in {u,p,z,q,v} {\fill (\w) circle (1.5pt);}
\node[left] at (u) {$u$};
\node[right] at (v) {$v$};
\node[above] at (p) {$p$};
\node[below] at (z) {$z$};
\node[above] at (q) {$q$};
\end{tikzpicture}
\par\smallskip\small (a) $D(u,v)$
\end{minipage}\hfill
\begin{minipage}[c]{.61\textwidth}
\centering
\begin{tikzpicture}[x=9.5mm,y=7.8mm,
 every node/.style={font=\scriptsize},
 list/.style={draw=black!65,fill=black!3,line width=.45pt},
 matching/.style={line width=.55pt}]
\draw[list] (-3.45,.10) ellipse (.90 and .82);
\draw[list] (-1.85,2.20) ellipse (.65 and .84);
\draw[list] (0,.10) ellipse (1.55 and 1.00);
\draw[list] (1.85,2.20) ellipse (.65 and .84);
\draw[list] (3.45,.10) ellipse (.90 and .82);
\node at (-3.45,-.97) {$L(u)$};
\node at (-1.85,3.27) {$L(p)$};
\node at (0,-1.15) {$L(z)$};
\node at (1.85,3.27) {$L(q)$};
\node at (3.45,-.97) {$L(v)$};
\node at (-3.45,-.46) {$\cdots$};
\node at (3.45,-.46) {$\cdots$};
\coordinate (u1) at (-3.8,0);
\coordinate (u2) at (-3.1,.38);
\coordinate (p1) at (-2,2.6);
\coordinate (p2) at (-2,1.75);
\coordinate (za1) at (-.45,0);
\coordinate (za2) at (-1.05,.38);
\coordinate (zb1) at (.45,0);
\coordinate (zb2) at (1.05,.38);
\coordinate (q1) at (2,2.6);
\coordinate (q2) at (2,1.75);
\coordinate (v1) at (3.8,0);
\coordinate (v2) at (3.1,.38);
\draw[matching] (u1)--(p1)--(za1)--cycle;
\draw[matching] (u2)--(p2)--(za2)--cycle;
\draw[matching] (v1)--(q1)--(zb1)--cycle;
\draw[matching] (v2)--(q2)--(zb2)--cycle;
\foreach \w in {u1,u2,p1,p2,za1,za2,zb1,zb2,q1,q2,v1,v2}
 {\fill (\w) circle (1.35pt);}
\node[below=1pt] at (u1) {$a_1$};
\node[above=1pt] at (u2) {$a_2$};
\node[right=1pt] at (p1) {$\alpha_1$};
\node[right=1pt] at (p2) {$\alpha_2$};
\node[below=1pt] at (za1) {$\alpha_1$};
\node[above=2pt,fill=black!3,inner sep=.5pt] at (za2) {$\alpha_2$};
\node[below=1pt] at (zb1) {$\beta_1$};
\node[above=2pt,fill=black!3,inner sep=.5pt] at (zb2) {$\beta_2$};
\node[left=1pt] at (q1) {$\beta_1$};
\node[left=1pt] at (q2) {$\beta_2$};
\node[below=1pt] at (v1) {$b_1$};
\node[above=1pt] at (v2) {$b_2$};
\end{tikzpicture}
\par\smallskip\small (b) A DP-cover of $D(u,v)$
\end{minipage}
\caption{The graph $D(u,v)$ and the DP-cover used in
Lemma~\ref{lem:D}. Only edges between different lists are shown
in~(b).}
\label{fig:D-cover}\label{fig:lower-D}
\end{figure}

\begin{lemma}\label{lem:D}
Let $L$ assign pairwise disjoint lists to the vertices of $D(u,v)$,
with $|L(w)|=d_{D(u,v)}(w)$ for every internal vertex $w$.
For any $A\subseteq L(u)$ and $B\subseteq L(v)$ with
$|A|=|B|=2$, there is a DP-cover $(L,M)$ of $D(u,v)$ such that
no choice of colours from $A$ at $u$ and $B$ at $v$ extends to
an $(L,M)$-colouring.
\end{lemma}

\begin{proof}
Write $A=\{a_1,a_2\}$ and $B=\{b_1,b_2\}$. Construct the cover
shown in Figure~\ref{fig:D-cover}(b) as follows. Label the colours
in $L(p)$ by $\alpha_1,\alpha_2$, those in $L(q)$ by
$\beta_1,\beta_2$, and those in $L(z)$ by
$\alpha_1,\alpha_2,\beta_1,\beta_2$. For $i=1,2$, match $a_i$
to $\alpha_i$ on each of $up,uz$, and match the two copies of
$\alpha_i$ on $pz$. Similarly, match $b_i$ to $\beta_i$ on each
of $vq,vz$, and match the two copies of $\beta_i$ on $qz$.
Leave all other pairs unmatched. The internal list sizes are
$2,4,2$, as required.

If $u$ receives $a_i$, then $p$ must receive $\alpha_{3-i}$.
The choices at $u$ and $p$ therefore forbid $\alpha_i$ and
$\alpha_{3-i}$ at $z$, respectively. Similarly, a choice from $B$
at $v$ forces a colour at $q$ and forbids both $\beta$ colours at
$z$. Thus no colour remains at $z$.
\end{proof}

We attach copies of $D$ to a fan to construct the graph for the
lower bound. For integers $r\ge1$ and $N\ge2$, start with a path
$v_0v_1\cdots v_N$ and a vertex $x$ adjacent to every path vertex.
For each $i\in\{0,\ldots,N-1\}$, attach $r$ copies of
$D(v_i,v_{i+1})$ with pairwise disjoint interiors, retaining all
edges of the fan. Denote the resulting graph by $F_{r,N}$, and
call $x,v_0,\ldots,v_N$ its \emph{core vertices}. We call the graph induced by core vertices the fan of $F_{r,N}$.
Figure~\ref{fig:lower-fan} shows $F_{2,4}$.

\begin{figure}[htbp]
\centering
\begin{tikzpicture}[x=6mm,y=7mm,
 every node/.style={font=\scriptsize}]
\coordinate (x) at (6,2.5);
\fill (x) circle (1.6pt);
\node[above] at (x) {$x$};
\foreach \i in {0,...,4} {
 \coordinate (v\i) at (3*\i,0);
 \draw[thick] (x)--(v\i);
 \fill (v\i) circle (1.6pt);
 \node[above=1mm,fill=white,inner sep=1pt] at (v\i) {$v_{\i}$};
}
\foreach \i in {0,...,3} {
 \pgfmathtruncatemacro{\j}{\i+1}
 \draw[thick] (v\i)--(v\j);
 \foreach \b/\zz/\pp in {1/-.45/-.60,2/-1.65/-2.00} {
  \coordinate (z\i\b) at (3*\i+1.5,\zz);
  \coordinate (p\i\b) at (3*\i+.7,\pp);
  \coordinate (q\i\b) at (3*\i+2.3,\pp);
  \draw[gray!80,line width=.55pt] (v\i)--(p\i\b)--(z\i\b)--(v\i)
                (v\j)--(q\i\b)--(z\i\b)--(v\j);
  \foreach \w in {z,p,q} {\fill (\w\i\b) circle (1.3pt);}
 }
}
\end{tikzpicture}
\caption{The graph $F_{2,4}$.}
\label{fig:lower-fan}
\end{figure}
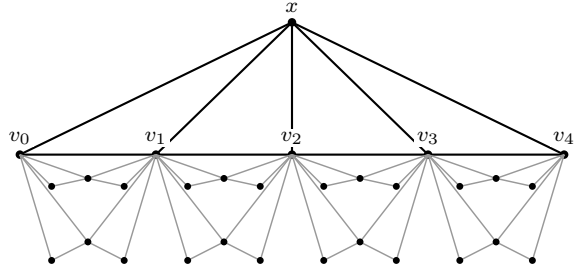

\begin{lemma}\label{lem:lower-fan-minor}
For all integers $r\ge1$ and $N\ge2$, the graph $F_{r,N}$ belongs
to $\mathcal S_{2r+1}$.
\end{lemma}
\begin{proof}
The graph is clearly $2$-connected and not a cycle. It is also
$K_4$-minor-free, since it is a $2$-tree: the fan is a $2$-tree,
and each copy of $D(u,v)$ is added by joining $z$ to $u,v$, then
$p$ to $u,z$ and $q$ to $v,z$.

For $r=1$, the graph is easily seen to be outerplanar, and hence
has no $K_{2,3}$ minor.
For $r\ge2$, there are $r$ copies attached to each path edge $v_{i-1}v_{i}$, where $i\in [N]$. 
Let $H_j$ consist of the fan and the $j$th copy on each path edge,
for $j\in[r]$. Each $H_j$ is isomorphic to $F_{1,N}$ and is
therefore outerplanar.

Suppose there are pairwise disjoint connected vertex sets
$A,B,C_1,\ldots,C_{2r+1}$, with each $C_i$ adjacent to both
$A$ and $B$. A connected set avoiding the core vertices lies inside one
copy of $D$ and has at most four external neighbours. Since
$2r+1\ge5$, both $A$ and $B$ must contain some core vertices.

For each connected set containing core vertices, we show that
its intersection with $V(H_j)$ preserves connectivity and adjacency
between disjoint sets. Different copies of $D$ are linked through
core vertices, and the terminals of each copy are adjacent in the
fan. Since $H_j$ contains the whole fan, each intersection remains
connected. If an edge joining two such sets lies in an omitted
copy of $D(v_{i-1},v_{i})$, then each set must contain a terminal of $D(v_{i-1},v_{i})$. As the sets are disjoint, one contains $v_{i-1}$
and the other contains $v_i$, so the edge $v_{i-1}v_i$ preserves
their adjacency.

For each $C_i$ containing a core vertex, we now know that its connectivity and required adjacencies are retained in every $H_j$. Any other $C_i$ lies wholly inside one copy of $D$, it still retains these properties in the $H_j$
containing that copy. There are $2r+1$ sets $C_i$ and only $r$
graphs $H_j$, so in some $H_j$ at least three intersections
$C_i\cap V(H_j)$ are nonempty. Together with $A\cap V(H_j)$ and
$B\cap V(H_j)$, three of these connected sets give a $K_{2,3}$
minor, contradicting outerplanarity.
\end{proof}

To construct an uncolourable DP-cover of $F_{r,N}$, label the
colours in each core list by $[k+1]$ and match equal labels on
each edge $xv_i$.
If $x$ receives colour $a$, then every path vertex must avoid $a$.
We use a separate path segment for each $a\in[k+1]$, and choose
its cover, including the attached copies of $D$, to be uncolourable
with the remaining $k$ colours. The following lemma gives this
construction.

\begin{lemma}\label{lem:lower-paththreshold}
For $k\ge3$, put
\[
 R(k)=\left\lceil\frac{(k+1)^2-4}{16}\right\rceil.
\]
Let $H$ be obtained from a path $P$ with $k$ edges by attaching
$r\ge R(k)$ copies of $D(u,v)$ to each edge $uv$ of $P$, with
pairwise disjoint interiors.
Then $H$ is not degree-truncated DP-$k$-colourable.
\end{lemma}
\begin{proof}
For each $v\in V(P)$, label the colours in $L(v)$ by $[k]$.
Give the internal vertices $p,z,q$ of each copy of $D$ lists of
sizes $2,4,2$, respectively. Along successive edges of $P$, we
use the edge matching and attached copies of $D$ to forbid all
pairs in $[m]\times\{m,m+1,\ldots,k\}$, for
$m=k,k-1,\ldots,1$.
The available colours are then forced into successively smaller
sets until none remain.

A path edge supplies a matching of forbidden pairs, and
Lemma~\ref{lem:D} allows each attached copy of $D$ to forbid a
prescribed $K_{2,2}$. For positive integers $m,n$, call a complete
bipartite subgraph of $K_{m,n}$ with one or two vertices in each
part a \emph{small rectangle}.
We show that whenever $m+n=k+1$, the edges of $K_{m,n}$ can be
covered by at most $R(k)$ small rectangles and a matching
(possibly empty). These rectangles may overlap.

Assume $k=2h-1$. Since $h^2\equiv0 \text{ or } 1\pmod4$, $R(k)=\lfloor h^2/4\rfloor$.
If $m=2a$ and $n=2b$, partition the two parts of $K_{m,n}$ into
$a$ pairs and $b$ pairs, respectively. Each choice of one pair
from each part gives a $K_{2,2}$, and these $ab$ copies cover all
edges. Since $a+b=h$, we have $ab\le\lfloor h^2/4\rfloor$.

Now let $m=2a+1\le n=2b+1$.  We first cover $K_{2a+1,2a+1}$ by
$a(a+1)$ rectangles and a matching, by induction on $a$.
The case $a=0$ needs no rectangles. For $a\ge1$, apply induction to $K_{2a-1,2a-1}$, and add two new colours $1,2$ to the first part $A$ and two new colours $3,4$ to
the second part $B$.
Partition $B\setminus\{3\}$ into $a$ pairs and combine each
pair with $\{1,2\}$ to form a $K_{2,2}$.
Similarly, partition $A\setminus\{2\}$ into $a$ pairs and
combine each pair with $\{3,4\}$.
These $2a$ rectangles cover all edges incident with the new
colours except $(2,3)$. Adding this pair to the old matching
still gives a matching in $K_{2a+1,2a+1}$.
This completes the induction.

Apply this construction to a
$K_{2a+1,2a+1}$ containing the whole smaller part, and pair the
$2(b-a)$ remaining colours of the larger part. For each such
pair, partition the smaller part into $a$ pairs and one singleton;
the edges between them are covered by $a+1$ small rectangles.
The total is
\[
 a(a+1)+(b-a)(a+1)=b(a+1)\le\lfloor h^2/4\rfloor,
\]
as $b+(a+1)=h$.

Assume $k=2h$. Write the parts as $U,V$, with
$|U|=2a$, $|V|=2b+1$, where $a\ge1$, $b\ge0$ and $a+b=h$.
Partition $U$ into $a$ pairs and $V$ into $b$ pairs and a single
colour $\gamma$. Each pair in $U$ requires $b$ copies of
$K_{2,2}$ and one $K_{2,1}$ to cover its edges to $V$.
This gives $a(b+1)$ rectangles.

We can save one rectangle as follows. Take two pairs from $U$,
labelled $\{1,2\},\{3,4\}$, and one pair $\{\alpha,\beta\}$
from $V$. The initial cover uses four rectangles between $\{1,2,3,4\}$ and $\{\alpha,\beta,\gamma\}$: two $K_{2,2}$s
and two $K_{2,1}$s. Replace them by three $K_{2,2}$s with
respective parts
\[
 \{1,2\},\ \{\gamma,\alpha\};\qquad
 \{2,3\},\ \{\alpha,\beta\};\qquad
 \{3,4\},\ \{\gamma,\beta\}.
\]
Only the edges $(1,\beta)$ and $(4,\alpha)$ remain uncovered,
and they form a matching. Thus this replacement saves one
rectangle.

Each replacement uses two pairs from $U$ and one pair from $V$.
Without reusing any of these pairs, we can make
$q=\min\{\lfloor a/2\rfloor,b\}$ replacements.
All replacements use the same colour $\gamma$, but no uncovered
edge meets $\gamma$. Hence all uncovered edges together form
a matching.
The construction uses
\[
 a(b+1)-q=\max\{ab+\lceil a/2\rceil,\,(a-1)(b+1)+1\}
\]
rectangles.
Both terms are at most $\lceil h(h+1)/4\rceil=R(k)$, since
\[
 \begin{aligned}
 ab+\lceil a/2\rceil
 &=\left\lceil\frac{(2a)(2b+1)}4\right\rceil
 \le\left\lceil\frac{h(h+1)}4\right\rceil,\\
 (a-1)(b+1)+1
 &\le\lfloor h^2/4\rfloor+1
 \le\lceil h(h+1)/4\rceil\qquad(h\ge2).
 \end{aligned}
\]

Therefore, the union of  at most $R(k)$ small rectangles
and one matching covers
$[m]\times\{m,m+1,\ldots,k\}$ for $m=k,k-1,\ldots,1$.
Place the matching on the corresponding path edge. Enlarge each
singleton side of a rectangle to two colours from $[k]$, and use Lemma~\ref{lem:D} to forbid the resulting
$K_{2,2}$ in a separate attached copy of $D$. Leave all matchings
in unused copies empty.
In any colouring of this cover, the successive path vertices
would have to receive colours in
$[k],[k-1],\ldots,[1],\varnothing$. Thus $H$ is not
degree-truncated DP-$k$-colourable.
\end{proof}

\begin{theorem}\label{thm:lower-fan}
For every integer $t\ge3$, with
$r=\lfloor(t-1)/2\rfloor$,
\[
 \chi^{\text{\sout{$\mathrm{d}$}}}_{\DP}(\mathcal S_t)
 \ge 1+\left\lfloor\sqrt{16r+4}\right\rfloor.
\]
\end{theorem}
\begin{proof}
Put $k=\lfloor\sqrt{16r+4}\rfloor-1$ and $N=k(k+1)$.
Then $k\ge3$ and $r\ge R(k)$. By
Lemma~\ref{lem:lower-fan-minor}, $F_{r,N}\in\mathcal S_t$,
since $2r+1\le t$.

Label the colours in each core list by $[k+1]$ and match equal
labels on each spoke. On each of the $k+1$ consecutive $k$-edge segments
and its attached copies of $D$, use the matchings constructed in
Lemma~\ref{lem:lower-paththreshold} on the path colours labelled
by $[k+1]\setminus\{a\}$, one segment for each $a\in[k+1]$.
Any choice of colour $a$ at $x$ leaves the corresponding segment
uncolourable.

Since $2r+2\ge k+1\ge4$, the core lists and the internal lists
of sizes $2,4,2$ have exactly the required truncated sizes. Hence
$
 \chi^{\text{\sout{$\mathrm{d}$}}}_{\DP}(\mathcal S_t)
 \ge k+2=1+\lfloor\sqrt{16r+4}\rfloor$.
\end{proof}

\section{A square-root upper bound}\label{sec:upper}

In this section, we will prove the following upper bound.

\begin{theorem}\label{thm:upper}
For every integer $t\ge3$, with $r=\lfloor(t-1)/2\rfloor$,
\[
 \chi^{\text{\sout{$\mathrm{d}$}}}_{\DP}(\mathcal S_t)\le\lceil8\sqrt r\rceil.
\]
\end{theorem}

\begin{lemma}\label{lem:degree-two}
Let $G$ be a simple $2$-connected $K_4$-minor-free graph and let
$uv\in E(G)$. Then $G$ has a degree-two vertex outside $\{u,v\}$.
\end{lemma}
\begin{proof}
We prove this by induction on $|V(G)|$. The triangle is immediate. Every $3$-connected graph on at least four vertices contains a subdivision of $K_4$:
choose a vertex $w$, a cycle in $G-w$, and three internally disjoint
paths from $w$ to distinct vertices of that cycle. Such paths exist by Menger's theorem and  $3$-connectivity.

Thus, if $|V(G)|\ge4$, there is a two-vertex cut $\{a,b\}$. Assume $C_1,...,C_k$ ($k\ge 2$) are connected components of $G-\{a,b\}$. If $\{a,b\}=\{u,v\}$, then apply induction to $G[V(C_1)\cup\{a,b\}]$. We know that $G[V(C_1)\cup\{a,b\}]$ has a degree-two vertex in $C_1$, which is also a degree-two vertex in $G$. 

We may assume that $u, v \in V(C_1)\cup\{a,b\}$. Choose another component $C_2$.  Let $G'=G[V(C_2)\cup\{a,b\}]$ if $ G[V(C_2)\cup\{a,b\}]$ contains the edge $ab$, otherwise, $G'=G[V(C_2)\cup\{a,b\}]+ab$. Since $ab\in E(G')$,  the  graph $G'$ is 2-connected and $K_4$-minor free.  By induction on $G'$, we know that there is a degree-two vertex outside $\{a,b\}$. The vertex has the same degree in $G$ and is distinct from $u,v$.
\end{proof}

After removing a degree-two vertex $z$ with neighbours $x,y$,
we must forbid any pair of colours at $x,y$ that cannot be
extended to a colour at $z$. Together with any existing restrictions
between $x$ and $y$, these pairs need not form a matching.
We therefore allow more general colour constraints.
We use $S$ for the vertices that still need to be coloured.
Vertices outside $S$ are kept only to count two-edge paths
between vertices in $S$.

Let $G$ be a simple graph, and let $S\subseteq V(G)$ with $|S|\ge2$.
Assume that $V(G)\setminus S$ is independent and every vertex outside
$S$ has degree two. Thus each such vertex lies on a two-edge path
whose ends belong to $S$. For distinct $x,y\in S$, define 
\begin{align*}
 \mu_{xy}&=|\{z\in V(G)\setminus S:N_G(z)=\{x,y\}\}|,
 \\
 \delta_{xy}&=\begin{cases}1,&xy\in E(G),\\0,&xy\notin E(G),\end{cases}
 \qquad
 P(G,S)=\{xy:\mu_{xy}+\delta_{xy}>0\}.
\end{align*}
Here $xy$ denotes an unordered pair. In particular, $\mu_{xy}$ is
the number of two-edge $x$--$y$ paths with their internal vertex
outside $S$, we call it \emph{path multiplicity}. The path multiplicity is determined by $G$ and $S$.

We assign lists only to the vertices in $S$. Initially $S=V(G)$, so every $\mu_{xy}$ is zero and $P(G,S)=E(G)$.
Figure~\ref{fig:marked-set} shows how $S$ and the multiplicities can
change even when the graph $G$ stays the same.

\begin{figure}[H]
\centering
\begin{tikzpicture}[x=6mm,y=6mm,every node/.style={font=\footnotesize},
 active/.style={circle,fill=black,inner sep=1.8pt},
 inactive/.style={circle,draw=black,fill=white,inner sep=1.7pt}]
\foreach \panel in {0,1,2}{
 \begin{scope}[xshift=\panel*4.2cm]
  \coordinate (x) at (0,0);\coordinate (y) at (3,0);
  \foreach \i/\h in {1/1.8,2/1.1,3/.4,4/-.4,5/-1.1,6/-1.8}{
   \coordinate (z\i) at (1.5,\h);
   \draw (x)--(z\i)--(y);
   \ifnum\panel=0
    \node[active] at (z\i) {};
   \else\ifnum\panel=1
    \ifnum\i<3 \node[inactive] at (z\i) {};
    \else \node[active] at (z\i) {};\fi
   \else \node[inactive] at (z\i) {};\fi\fi
  }
  \node[active,label=left:$x$] at (x) {};
  \node[active,label=right:$y$] at (y) {};
  \ifnum\panel=0
   \node at (1.5,-2.4) {$S=V(G),\ \mu_{xy}=0$};
  \else\ifnum\panel=1
   \node at (1.5,-2.4) {$|S|=6,\ \mu_{xy}=2$};
  \else
   \node at (1.5,-2.4) {$S=\{x,y\},\ \mu_{xy}=6$};
  \fi\fi
 \end{scope}
}
\end{tikzpicture}
\caption{The same graph $G=K_{2,6}$ with three choices of $S$.
Filled vertices belong to $S$; open vertices are
outside $S$ and are counted by $\mu$.}
\label{fig:marked-set}
\end{figure}
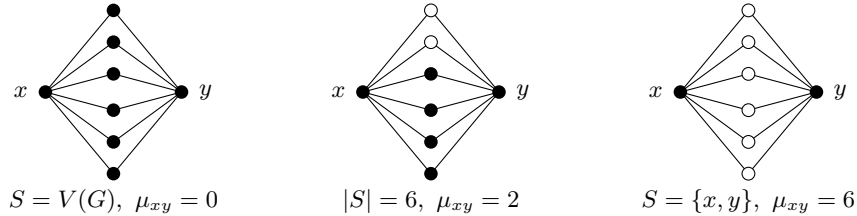

The cover of $D(x,y)$ in Lemma~\ref{lem:D} shows that deleting
internal vertices can leave a $K_{2,2}$ of forbidden pairs between
the terminal lists, even when all original constraints are matchings.
Our induction uses the same idea: when a vertex with two remaining
neighbours is deleted, we add edges between their lists for pairs
that cannot extend to that vertex. We therefore introduce
subgraphs of $K_{2,2}$ as basic constraints, together with
endpoint weights.

A \emph{small relation} between two lists is a bipartite graph whose
edges meet at most two colours in each list. We identify such a
graph with its edge set. For a small relation
$Q\subseteq L(x)\times L(y)$, put
\begin{align*}
 \lambda_Q(x)&=\max\bigl(\{1\}\cup\{d_Q(b):b\in L(y)\}\bigr),\\
 \lambda_Q(y)&=\max\bigl(\{1\}\cup\{d_Q(a):a\in L(x)\}\bigr).
\end{align*}
The weight $ \lambda_Q(x)$ at $x$ bounds the number of colours forbidden in $L(x)$
by one colour at $y$. An empty small relation has weight $1$ at both endpoints.
Each small relation is a subgraph of $K_{2,2}$, and has at most four edges. Summing degrees over
the at most two colours incident with its edges in either list gives $|E(Q)|\le2\lambda_Q(x)$ and $|E(Q)|\le2\lambda_Q(y)$.

Several such constraints may arise between the same two lists
as the induction proceeds. We therefore allow a family of small
relations $Q_i$ between $L(x)$ and $L(y)$, together with the
original matching on $xy$ if this edge is present.

With truncated lists, a later reduction may create forbidden
pairs involving many colours. We handle these by specifying
sets $A_{xy}(x)$ and $A_{xy}(y)$ of \emph{forbidden colours}.
These colours remain in the lists, but cannot be chosen.

We collect these data in the tuple
\begin{equation}\label{eq:tuple}
 \begin{aligned}
 \mathcal C&=\bigl(G,S,L,(\mathcal D_{xy})_{xy\in P(G,S)}\bigr),\\
 \mathcal D_{xy}&=\bigl(A_{xy}(x),A_{xy}(y),M_{xy},
          (Q_{xy,i})_{i=1}^{q_{xy}},\varepsilon_{xy}\bigr).
 \end{aligned}
\end{equation}
Its entries have the following meaning:
\begin{itemize}[leftmargin=*]
\item $L$ assigns pairwise disjoint finite lists $L(x)$ to $x\in S$.
\item $A_{xy}(x)\subseteq L(x)$ and $A_{xy}(y)\subseteq L(y)$.
\item $M_{xy}$ is a matching between the two lists. It is empty
if $\delta_{xy}=0$; it may also be empty if $\delta_{xy}=1$.
\item $q_{xy}$ is a nonnegative integer, and each $Q_{xy,i}$ is
a small relation, possibly empty.
\item $\varepsilon_{xy}\in\{0,1\}$, with $\varepsilon_{xy}=1$
whenever $A_{xy}(x)\cup A_{xy}(y)\ne\varnothing$.
\end{itemize}
The forbidden pairs at $xy$ are
\[
 F_{xy}=\bigl(A_{xy}(x)\times L(y)\bigr)
       \cup\bigl(L(x)\times A_{xy}(y)\bigr)
       \cup M_{xy}\cup\bigcup_{i=1}^{q_{xy}}Q_{xy,i}.
\]
A colouring of $\mathcal C$ chooses a colour in each $L(x)$,
$x\in S$, and avoids $F_{xy}$ for every $xy\in P(G,S)$.

Define
\[
 w_{xy}(x)=|A_{xy}(x)|+
       \sum_{i=1}^{q_{xy}}\lambda_{Q_{xy,i}}(x)
       +\delta_{xy}+\varepsilon_{xy},
\]
and define $w_{xy}(y)$ symmetrically. The weight at $x$ bounds
how many colours in $L(x)$ a choice at $y$ can forbid. More precisely,
for every $b\in L(y)\setminus A_{xy}(y)$,
\[
 d_{F_{xy}}(b)\le |A_{xy}(x)|+
 \sum_{i=1}^{q_{xy}}\lambda_{Q_{xy,i}}(x)+\delta_{xy}
 =w_{xy}(x)-\varepsilon_{xy}.
\]
Here the terms account for the forbidden colours at $x$, the small
relations, and the matching, respectively. Thus, when
$A_{xy}(x)\ne\varnothing$, we have $\varepsilon_{xy}=1$, and
any such choice at $y$ forbids at most $w_{xy}(x)-1$ colours at $x$.
The extra unit also gives $w_{xy}(y)\ge |A_{xy}(y)|+1$, so
$|L(y)|\ge w_{xy}(y)$ guarantees a colour outside $A_{xy}(y)$.
The same statements hold with $x$ and $y$ interchanged.
In the truncated setting of Theorem~\ref{thm:valid}, the existence
of such a colour is established in Claim~\ref{clm:weight-bounds}.

Fix a threshold $h>0$. In the union of $q$ small relations,
we add every colour of degree greater than $h$ to the corresponding
forbidden set. This union has at most $4q$ edges.  Hence at most $4q/h$ colours are added to the forbidden set at each end.
We charge $h/4$ units of multiplicity per newly forbidden colour.
The total charge at each end is at most $q$.
We impose the same charge on previously forbidden colours.
When the multiplicity is at most $r$, this bounds their number
at either end by $4r/h$. Each colour outside these sets has at most
$h$ neighbours in the small relations. The two main terms $4r/h$
and $h$ balance at $h=2\sqrt r$. With this choice, the
$\mu-q$ units not assigned to small relations allow at most
$2(\mu-q)/\sqrt r$ forbidden colours at each end, as required below.

\Needspace{12\baselineskip}
\begin{definition}\label{def:valid}
Let $r\ge1$ and $k\ge1$ be integers. The tuple $\mathcal C$ is
\emph{$(r,k)$-valid} if, for every $xy\in P(G,S)$,
\begin{align}
 &|A_{xy}(v)|\le\frac{2(\mu_{xy}-q_{xy})}{\sqrt r}
       \qquad(v\in\{x,y\}),
 \notag\\
 &\mu_{xy}=q_{xy}\quad\hbox{if }\varepsilon_{xy}=0,
 \label{eq:normalisation}
\end{align}
and, for every $x\in S$,
\begin{equation}\label{eq:lists}
 |L(x)|\ge\min\left\{k,\sum_{y:\,xy\in P(G,S)}w_{xy}(x)\right\}.
\end{equation}
\end{definition}

If a pair $xy\in P(G,S)$ has neither a small relation nor a direct
edge, validity forces $\varepsilon_{xy}=1$. Indeed, otherwise
$\mu_{xy}=q_{xy}=0$, contradicting $xy\in P(G,S)$.
Thus $q_{xy}+\delta_{xy}+\varepsilon_{xy}>0$ at every pair.

A simple DP-cover $(L,M)$ is represented by taking $S=V(G)$, using its
edge matchings, and putting $q_{xy}=\varepsilon_{xy}=0$ and
$A_{xy}(x)=A_{xy}(y)=\varnothing$.
Then $\mu_{xy}=0$, $\delta_{xy}=1$ and $w_{xy}(x)=1$ on each
edge. Thus \eqref{eq:lists} is exactly the degree-truncated
list condition. Equation~\eqref{eq:normalisation} also says that,
when $\varepsilon_{xy}=0$, the number of small relations is
the actual path multiplicity in $G$.

The following strengthened theorem also allows a cycle when some constraint
has $\varepsilon=1$. This extra case is needed because a reduction
may produce a cycle.

\Needspace{8\baselineskip}
\begin{theorem}\label{thm:valid}
Let $r\ge1$ be an integer and put $k=\lceil8\sqrt r\rceil$.
Let $\mathcal C=(G,S,L,(\mathcal D_{xy})_{xy\in P(G,S)})$ be an
$(r,k)$-valid tuple, where $|S|\ge2$ and $V(G)\setminus S$ is
an independent set of degree-two vertices.
Suppose $G$ is $2$-connected and
$\{K_4,K_{2,2r+2}\}$-minor-free. If $G$ is not a cycle, or $G$ is a cycle with some
$\varepsilon_{xy}=1$, then $\mathcal C$ has a colouring.
\end{theorem}
\begin{proof}
Suppose not, and choose an uncolourable tuple satisfying the
hypotheses with $|S|$ minimum. We may assume $|L(x)|\le k$ for
every $x\in S$: restricting larger lists to $k$ colours and
restricting the constraints accordingly preserves validity.

For a pair $xy\in P(G,S)$, write $F$, $A(x)$, $A(y)$, $q$,
$\delta$, $\varepsilon$, $\mu$, $w(x)$ and $w(y)$ for its data,
and put $R=\bigcup_i Q_i$. Let
\[
 T(x)=\{a\in L(x)\setminus A(x):d_R(a)>2\sqrt r\},
\]
and define $T(y)$ similarly.

\begin{claim}\label{clm:weight-bounds}
For every pair $xy\in P(G,S)$, there is a colour in
$L(x)\setminus(A(x)\cup T(x))$. Moreover,
\[
 |A(x)|+|T(x)|\le\frac{2\mu}{\sqrt r},\qquad
 |A(x)|+|T(x)|+1\le |L(x)|.
\]
The same assertions hold at $y$.
\end{claim}
\begin{proof}
Each colour in $T(x)$ has more than $2\sqrt r$ neighbours in $R$,
which has at most $4q$ edges. Counting edges at $T(x)$ gives
$|T(x)|\le2q/\sqrt r$. Combining this with
$|A(x)|\le2(\mu-q)/\sqrt r$ proves the first bound.

Put $s=\sum_i\lambda_{Q_i}(x)$.
Since $|E(Q_i)|\le2\lambda_{Q_i}(x)$, we have $|E(R)|\le2s$.
If $s>0$, the strict threshold in the definition of $T(x)$ gives
$|T(x)|<s/\sqrt r\le s$. By integrality, $|T(x)|+1\le s$, so
$w(x)\ge |A(x)|+|T(x)|+1$.
If $s=0$, then $q=0$, $T(x)=\varnothing$ and
$\delta+\varepsilon\ge1$, giving the same bound.

The $\mu$ common neighbours of $x,y$ give a $K_{2,\mu}$ subgraph
of $G$, so $\mu\le2r+1$. Therefore
\[
 |A(x)|+|T(x)|\le\frac{2\mu}{\sqrt r}
 \le4\sqrt r+\frac2{\sqrt r}<8\sqrt r\le k.
\]
Together with the list condition, this gives $|L(x)|\ge\min\{k,w(x)\}\ge |A(x)|+|T(x)|+1$,
so a colour outside $A(x)\cup T(x)$ exists.
\end{proof}

\begin{claim}\label{clm:degree-bound}
For every $a\in L(x)\setminus A(x)$,  $d_F(a)\le w(y)-\varepsilon$.
If also $a\notin T(x)$, then
\[
 d_F(a)\le\frac{2\mu}{\sqrt r}+2\sqrt r-1.
\]
In particular, when $\mu\le r$, such a colour forbids at most
$4\sqrt r-1$ colours in $L(y)$.
The same assertions hold with $x,y$ interchanged.
\end{claim}
\begin{proof}
A colour $a\in L(x)\setminus A(x)$ forbids the colours in $A(y)$,
at most $\delta$ colours through the matching, and at most
$\lambda_{Q_i}(y)$ through each $Q_i$. Hence
\[
 d_F(a)\le |A(y)|+\delta+\sum_i\lambda_{Q_i}(y)
 =w(y)-\varepsilon.
\]

If $a\notin T(x)$, then $d_R(a)\le2\sqrt r$.
Each small relation gives $a$ at most two neighbours, so also
$d_R(a)\le2q$. Validity gives $|A(y)|\le2(\mu-q)/\sqrt r$.
Therefore
\begin{align*}
 d_F(a)
 &\le |A(y)|+\min\{2\sqrt r,2q\}+1\\
 &\le\frac{2(\mu-q)}{\sqrt r}+\min\{2\sqrt r,2q\}+1
 \le\frac{2\mu}{\sqrt r}+2\sqrt r-1.
\end{align*}
For the last inequality, the expression
$-2q/\sqrt r+\min\{2\sqrt r,2q\}+1$ is nondecreasing up to
$q=\sqrt r$ and decreases thereafter, so it is at most
$2\sqrt r-1$.
\end{proof}

\begin{claim}\label{clm:at-least-three}
$|S|\ge3$.
\end{claim}
\begin{proof}
Suppose $S=\{x,y\}$.
It suffices to find a pair of colours outside $F$.
Every other vertex is adjacent to $x,y$, so $G$ consists of
$\mu$ two-edge $x$--$y$ paths and the edge $xy$ when $\delta=1$.
These paths contain a $K_{2,\mu}$, so $\mu\le2r+1$.

First suppose one list has $k$ colours, say $|L(y)|=k$.
Claim~\ref{clm:weight-bounds} gives a colour
$a\in L(x)\setminus(A(x)\cup T(x))$.
Claim~\ref{clm:degree-bound} bounds the number of colours it forbids in $L(y)$.
Using $\mu\le2r+1$, we obtain
\[ d_F(a)
 \le\frac{2\mu}{\sqrt r}+2\sqrt r-1\le 6\sqrt r+\frac2{\sqrt r}-1<8\sqrt r\le k.
\]
Hence some colour in $L(y)$ is not adjacent to $a$ in $F$.

We may now assume that
$|L(x)|\ge w(x)$ and $|L(y)|\ge w(y)$.
If $\varepsilon=1$, Claim~\ref{clm:weight-bounds} gives a colour
$a\in L(x)\setminus A(x)$. By Claim~\ref{clm:degree-bound},
$d_F(a)\le w(y)-1<|L(y)|$, again leaving a nonadjacent colour in $L(y)$.

Finally, suppose $\varepsilon=0$. Validity gives $q=\mu$ and
$A(x)=A(y)=\varnothing$. Since $G$ is not a cycle, its number
of $x$--$y$ paths satisfies $q+\delta=\mu+\delta\ge3$.
Thus a small relation $Q$ exists and $|L(x)|\ge w(x)\ge3$.
As $Q$ meets at most two colours in $L(x)$, choose
$a\in L(x)$ incident with no edge of $Q$.
So
\[
 d_F(a)\le w(y)-\lambda_Q(y)\le w(y)-1<|L(y)|.
\]
Hence $a$ has a nonadjacent colour in $L(y)$, a contradiction.
\end{proof}

Let $H$ be the simple graph with vertex set $S$ and edge set
$P(G,S)$. We obtain $H$ from $G$ by replacing each two-edge path
through a vertex outside $S$ with an edge between its ends,
keeping only one copy of each edge. Hence $H$ is a
$2$-connected $K_4$-minor-free graph.
Choose an edge $e$ of $H$ with $\mu_e>r$ if one exists.
Otherwise choose an edge with $\mu_e+\delta_e\ge2$ if possible,
and choose any edge if neither condition applies.

\begin{claim}\label{clm:delete-vertex}
There is a degree-two vertex $z$ of $H$ outside the ends of $e$
whose neighbours $x,y$ satisfy $\mu_{xz},\mu_{zy}\le r$.
\end{claim}
\begin{proof}
By Lemma~\ref{lem:degree-two}, $H$ has a degree-two vertex $z$
outside the ends of $e$.

We show that at most one pair in $P(G,S)$ has multiplicity larger
than $r$. Suppose instead that $\mu_{e_1},\mu_{e_2}>r$ for two
distinct edges $e_1,e_2$ of $H$. Since $H$ is $2$-connected,
these edges lie on a common cycle $C$: subdivide the two edges
and apply Menger's theorem to the new vertices.
Removing $e_1,e_2$ from $C$ leaves two disjoint paths, one of
which may be a single vertex. Replace their edges by the
corresponding paths in $G$, and contract each of the two resulting
paths to a vertex. The $\mu_{e_1}+\mu_{e_2}\ge2r+2$ vertices
counted by these multiplicities are now common neighbours of
the two resulting vertices. Thus $G$ has a $K_{2,2r+2}$ minor,
a contradiction.

Neither edge incident with $z$ in $H$ is $e$, so both have
multiplicity at most $r$ by the choice of $e$.
\end{proof}

Fix $z,x,y$ as in Claim~\ref{clm:delete-vertex}, and put
$S'=S\setminus\{z\}$. To apply minimality, we will replace $G$
by a minor $G'$ and the two constraints at $z$ by one between
$x,y$. The new multiplicities will be computed from $G'$ and $S'$.

Put $X=L(x)$, $Z=L(z)$ and $Y=L(y)$.
A pair of colours at $x,y$ fails to extend to $z$ precisely when
the two colours together forbid all of $Z$. Let
\[
 \mathcal B=\{(a,b)\in X\times Y:
                   N_{F_{xz}}(a)\cup N_{F_{zy}}(b)=Z\}.
\]
We must therefore forbid every pair in $\mathcal B$ in the
replacement constraint. Its endpoint weights must not increase,
so that the lists at $x,y$ remain large enough.
We first construct this constraint with a multiplicity $\nu$,
then obtain $G'$ with $\nu$ additional two-edge $x$--$y$ paths.
The paths counted by $\mu_{xy}$ are retained, giving
$\mu'_{xy}=\mu_{xy}+\nu$.

\Needspace{9\baselineskip}
\begin{claim}\label{clm:replacement}
There are an integer $1\le\nu\le\max\{1,\mu_{xz},\mu_{zy}\}$
and a replacement constraint between $X,Y$ whose forbidden
relation $\widehat F_{xy}$ contains $\mathcal B$.
We use hats for the data of this additional constraint.
These data satisfy the conditions on multiplicities and forbidden colours in
Definition~\ref{def:valid} with multiplicity $\nu$ and $\widehat\delta=0$, and
\[
 \widehat w(x)\le w_{xz}(x),\qquad \widehat w(y)\le w_{zy}(y),\qquad
 \widehat\varepsilon\ge\max\{\varepsilon_{xz},\varepsilon_{zy}\}.
\]
If $\widehat\varepsilon=0$, then $\widehat q=\nu=1$.
\end{claim}
\begin{proof}
The replacement has no matching, and we set $\widehat\delta=0$.

\emph{Case 1: $|Z|\ge w_{xz}(z)+w_{zy}(z)$ and
$\varepsilon_{xz}=\varepsilon_{zy}=0$.}
Validity implies that all sets of forbidden colours at $xz,zy$ are empty.
We show that $\mathcal B$ is a small relation, so it can serve
as the single small relation of the replacement constraint.

By Claim~\ref{clm:degree-bound}, a pair $(a,b)$ forbids at most
$w_{xz}(z)+w_{zy}(z)$ colours of $Z$.
Thus $\mathcal B$ is empty if $|Z|>w_{xz}(z)+w_{zy}(z)$.
Assume that $|Z|=w_{xz}(z)+w_{zy}(z)$.  Considering a pair $(a,b)$ in $\mathcal B$, $a$ must be incident with every $Q_{xz,i}$,
and $b$ with every $Q_{zy,i}$.
If both constraints $F_{xz}$, $F_{yz}$ have small relations, then at most two colours in each list are incident with every small relations.
If only the constraint at $xz$ has a small relation, then the constraint at $yz$  is a matching. There are at most two possible
colours $a$ at $L(x)$. Each $a$ leaves
exactly one colour of $Z$ available, which the matching $M_{zy}$
can pair with at most one colour $b$.
The case with $x,y$ interchanged is identical. 
If neither has a small relation, both are matchings; a pair in
$\mathcal B$ is possible only when $|Z|=2$, and $\mathcal B$
is a matching with at most two edges. In every case, $\mathcal B$ is a small relation.

Take $\widehat F_{xy}=\mathcal B$, with $\widehat q=\nu=1$,
$\widehat\varepsilon=0$ and empty sets of forbidden colours.
Since $\mathcal B$ is a small relation, $\widehat w(x)\le2$, so we only need
to consider $w_{xz}(x)=1$.
For any $b\in Y$, choose $c\in Z\setminus N_{F_{zy}}(b)$;
this is possible since $d_{F_{zy}}(b)\le w_{zy}(z)<|Z|$.
Every neighbour of $b$ in $\mathcal B$ must be adjacent to $c$
in $F_{xz}$. Since $w_{xz}(x)=1$, $c$ has at most one neighbour in $X$. Thus we know that $d_{\mathcal B}(b)\le1$, and hence $\widehat w(x)=1$.
The bound at $y$ follows symmetrically.

\emph{Case 2: $|Z|\ge w_{xz}(z)+w_{zy}(z)$ and at least one of
$\varepsilon_{xz},\varepsilon_{zy}$ is $1$.}
Here the existing forbidden colours suffice. Put $\widehat A(x)=A_{xz}(x)$
and $\widehat A(y)=A_{zy}(y)$. If $a\notin \widehat A(x)$ and
$b\notin \widehat A(y)$, Claim~\ref{clm:degree-bound} gives
\[
 |N_{F_{xz}}(a)\cup N_{F_{zy}}(b)|
 \le w_{xz}(z)-\varepsilon_{xz}+w_{zy}(z)-\varepsilon_{zy}
 <|Z|.
\]
The strict inequality holds because at least one of the two
values of $\varepsilon$ is $1$. Thus every pair in $\mathcal B$
uses a colour in $\widehat A(x)$ or $\widehat A(y)$. Take
\[
 \widehat F_{xy}=(\widehat A(x)\times Y)\cup(X\times \widehat A(y)),\qquad \widehat q=0,\qquad
 \widehat\varepsilon=1.
\]
The proof of Claim~\ref{clm:weight-bounds} gives
$\widehat w(x)=|A_{xz}(x)|+1\le w_{xz}(x)$, and similarly at $y$.
The old bounds on forbidden colours also give
$|\widehat A(x)|\le2\mu_{xz}/\sqrt r$ and
$|\widehat A(y)|\le2\mu_{zy}/\sqrt r$.
Choose
\[
 \nu=\max\left\{1,
   \left\lceil\frac{\sqrt r}{2}|\widehat A(x)|\right\rceil,
   \left\lceil\frac{\sqrt r}{2}|\widehat A(y)|\right\rceil\right\}.
\]
Since $\mu_{xz},\mu_{zy}$ are integers, these bounds imply
$1\le\nu\le\max\{1,\mu_{xz},\mu_{zy}\}$.
Moreover, $|\widehat A(x)|$, $|\widehat A(y)|\le2\nu/\sqrt r$,
so the multiplicity conditions hold with $\widehat q=0$.

\emph{Case 3: $|Z|=k$.}
We forbid the endpoint colours that have more than $2\sqrt r$
neighbours in the small relations. Claim~\ref{clm:degree-bound}
will then bound how many colours the remaining choices at $z$.
Let $R_{xz}=\bigcup_i Q_{xz,i}$ and $R_{zy}=\bigcup_i Q_{zy,i}$, and put
\begin{align*}
 T_x&=\{a\in X\setminus A_{xz}(x):d_{R_{xz}}(a)>2\sqrt r\},\\
 T_y&=\{b\in Y\setminus A_{zy}(y):d_{R_{zy}}(b)>2\sqrt r\}.
\end{align*}
Set $\widehat A(x)=A_{xz}(x)\cup T_x$, $\widehat A(y)=A_{zy}(y)\cup T_y$, and take
$\widehat F_{xy}=(\widehat A(x)\times Y)\cup(X\times \widehat A(y))$, with $\widehat q=0$ and
$\widehat\varepsilon=1$.
Since $\mu_{xz},\mu_{zy}\le r$, Claim~\ref{clm:degree-bound} shows that
any $a\notin \widehat A(x)$ and $b\notin \widehat A(y)$ each forbid
at most $4\sqrt r-1$ colours of $Z$. Hence
\[
 |N_{F_{xz}}(a)\cup N_{F_{zy}}(b)|\le8\sqrt r-2<k=|Z|.
\]
Thus $\widehat F_{xy}$ contains $\mathcal B$.

Since $|\widehat A(x)|=|A_{xz}(x)|+|T_x|$,
Claim~\ref{clm:weight-bounds} and its proof give the bounds
\[
 |\widehat A(x)|\le\frac{2\mu_{xz}}{\sqrt r},\qquad
 \widehat w(x)=|\widehat A(x)|+1\le w_{xz}(x).
\]
The bounds at $y$ follow in the same way.

Here $\widehat q=0$, so we need
$|\widehat A(x)|,|\widehat A(y)|\le2\nu/\sqrt r$. Same as Case 2, choose
$$ \nu=\max\left\{1,
   \left\lceil\frac{\sqrt r}{2}|\widehat A(x)|\right\rceil,
   \left\lceil\frac{\sqrt r}{2}|\widehat A(y)|\right\rceil\right\}.$$
Since $\mu_{xz},\mu_{zy}$ are integers, the bounds above give
$\nu\le\max\{1,\mu_{xz},\mu_{zy}\}$.
This completes the construction in all three cases.
\end{proof}

Fix a replacement constraint and $\nu$ as in
Claim~\ref{clm:replacement}. We now define $G'$ explicitly.

If $\mu_{xz}=\mu_{zy}=0$, then both edges $xz,zy$ exist and
$\nu=1$. In this case set $G'=G$; the vertex $z$ now lies
outside $S'$.
Otherwise, assume without loss of generality that $\mu_{xz}\ge\mu_{zy}$.
Then $\nu\le\mu_{xz}$. Keep $\nu$ two-edge $x$--$z$ paths
and one $z$--$y$ path of length one or two in $G$. Delete the internal
vertices of all unused two-edge paths between $x,z$ or $z,y$,
and any unused edges $xz,zy$. Contract the retained $z$--$y$
path to $y$, so that each retained path $xuz$ becomes $xuy$.
All other vertices and edges are retained.

Observe that the construction gives a $2$-connected minor $G'$ of $G$
containing $S'$, and that $V(G')\setminus S'$ is an independent set
of degree-two vertices. Moreover, $\mu'_{xy}=\mu_{xy}+\nu$ , $\delta'_{xy}=\delta_{xy}$, where the old values $\mu_{xy}$ and $\delta_{xy}$ are taken to be zero
if $xy\notin P(G,S)$. For all distinct $u,v\in S'$ with $uv\ne xy$,
we have $\mu'_{uv}=\mu_{uv}$ and $\delta'_{uv}=\delta_{uv}$.

Let $H'$ be the simple graph on $S'$ with edge set $P(G',S')$.
Define $\mathcal C'=(G',S',L|_{S'},\mathcal D')$ as follows.
At $xy$, set
\[
\begin{aligned}
 \mu'_{xy}&=\mu_{xy}+\nu, & \delta'_{xy}&=\delta_{xy},\\
 A'_{xy}(u)&=A_{xy}(u)\cup\widehat A(u)
     &&(u\in\{x,y\}),\\
 M'_{xy}&=M_{xy}, & q'_{xy}&=q_{xy}+\widehat q,\\
 \varepsilon'_{xy}&=\max\{\varepsilon_{xy},\widehat\varepsilon\}.
\end{aligned}
\]
Writing $\widehat Q_1,\ldots,\widehat Q_{\widehat q}$ for the
replacement small relations, set
\[
 Q'_{xy,i}=\begin{cases}
 Q_{xy,i},&1\le i\le q_{xy},\\
 \widehat Q_{i-q_{xy}},&q_{xy}<i\le q'_{xy}.
 \end{cases}
\]
If $xy\notin P(G,S)$, take the old sets and relation family to be
empty and $\mu_{xy}=\delta_{xy}=q_{xy}=\varepsilon_{xy}=0$.
The data at all other pairs in $P(G',S')$ are unchanged.
Thus $F'_{xy}=F_{xy}\cup\widehat F_{xy}$.

The pair conditions in Definition~\ref{def:valid} follow directly
from the definitions and Claim~\ref{clm:replacement}. Moreover,
\[
 w'_{xy}(x)\le w_{xy}(x)+\widehat w(x)
 \le w_{xy}(x)+w_{xz}(x),
\]
and similarly at $y$, where the old weight at $xy$ is taken
to be zero if $xy\notin P(G,S)$. Thus the total weight at each
remaining vertex does not increase. Since the lists are unchanged,
the list-size condition is preserved. Hence $\mathcal C'$ is
$(r,k)$-valid.

It remains to check the condition on cycles. Here we use that
$z$ was chosen outside the ends of the edge $e$.

\begin{claim}\label{clm:cycle-condition}
If $G'$ is a cycle, then some constraint of $\mathcal C'$
has $\varepsilon'=1$.
\end{claim}
\begin{proof}
Suppose $G'$ is a cycle and every constraint of
$\mathcal C'$ has $\varepsilon'=0$.
Claim~\ref{clm:replacement} then gives
$\varepsilon_{xz}=\varepsilon_{zy}=0$ and $\nu=1$.
Thus all old values are zero, so $G$ is not a cycle.
We show that $G'$ still has a vertex of degree at least three.

If $\mu_e+\delta_e\ge2$, write $e=uv$.
All the $u$--$v$ paths represented by $e$ are retained,
since $z\notin\{u,v\}$.
As $H$ is $2$-connected, $u$ has a neighbour other than $v$.
The corresponding path is retained if this neighbour is not $z$;
otherwise, it is replaced by the new two-edge path.
In either case, $u$ has an incident edge in $G'$ in addition to
those on the retained $u$--$v$ paths. Hence $d_{G'}(u)\ge3$,
so $G'$ is not a cycle.

Otherwise, the choice of $e$ gives $\mu_{e'}+\delta_{e'}=1$ for every
$e'\in P(G,S)$. Each edge of $H$ is then represented by exactly
one path of length one or two in $G$.
Let $P$ be the $x$--$y$ path obtained by joining the paths
corresponding to $xz$ and $zy$. Every vertex of $P$ other than
$x,y$ has degree two. Since $\nu=1$, our construction replaces
$P$ by a two-edge $x$--$y$ path. This preserves the degrees of
$x,y$ and leaves all vertices outside $P$ unchanged.
Thus every vertex of degree at least three in $G$ remains in
$G'$ with the same degree. Such a vertex exists because $G$ is
$2$-connected and is not a cycle. Hence $G'$ is not a cycle either, a contradiction.
\end{proof}

All hypotheses now hold for $\mathcal C'$, and $|S'|<|S|$.
By minimality, $\mathcal C'$ has a colouring. Its colours $a,b$
at $x,y$ avoid $\widehat F_{xy}$, and therefore avoid $\mathcal B$.
Choose $c\in Z\setminus\bigl(N_{F_{xz}}(a)\cup N_{F_{zy}}(b)\bigr)$ and
assign it to $z$. This satisfies both constraints at $z$.
Every original constraint not involving $z$ was retained in
$\mathcal C'$, so we have a colouring of $\mathcal C$,
a contradiction.
\end{proof}

\begin{proof}[Proof of Theorem~\ref{thm:upper}]
Put $k=\lceil8\sqrt r\rceil$ and let $(L,M)$ be a DP-cover of
$G\in\mathcal S_t$ with $|L(v)|\ge\min\{k,d_G(v)\}$ for every
vertex $v$.
Take $S=V(G)$ and use its matchings as the records, with
$q_{xy}=\varepsilon_{xy}=0$ and empty sets of forbidden colours.
As observed above, this is an $(r,k)$-valid tuple.
The graph has no $K_{2,2r+2}$ minor, since $t\le2r+2$,
and it is not a cycle. Theorem~\ref{thm:valid} gives an
$(L,M)$-colouring of $G$.
\end{proof}

\section*{Declarations}

\noindent\textbf{Funding}\quad
No funding was received for conducting this study.

\medskip
\noindent\textbf{Conflict of interest}\quad
The author declares no conflict of interest.

\end{document}